\documentclass[11pt,letterpaper]{amsart}
\usepackage[T1]{fontenc}
\usepackage{lmodern}
\usepackage{amsmath,amssymb,amsthm,mathtools,microtype}
\usepackage{needspace}
\usepackage{xcolor}
\usepackage[unicode,breaklinks=true,colorlinks=true,linkcolor=blue,urlcolor=blue,citecolor=blue]{hyperref}
\numberwithin{equation}{section}
\newtheorem{Theorem}{Theorem}[section]
\newtheorem{Proposition}[Theorem]{Proposition}
\newtheorem{Lemma}[Theorem]{Lemma}
\newtheorem{Corollary}[Theorem]{Corollary}

\newcommand{\andand}{\quad \text{and} \quad}
\newcommand{\asas}{\quad \text{as~} }
\newcommand{\R}{\mathbb R}
\newcommand{\cH}{\mathcal{H}}

\def\indeq{\quad{}}
\def\llabel#1{\marginnote{\color{lightgray}\rm\small(#1)}[-0.0cm]\notag}
\usepackage{marginnote}
\def\comma{ {\rm ,\quad{}} }            
\def\colb{\color{black}}

\def\enew{\color{black}}

   \DeclareMathOperator{\supp}{supp}

\begin{document}
\title{Quantitative parabolic decay with bounded drift}
\author{Igor~Kukavica and Qi Xu}
\date{}
\begin{abstract}
A bounded time-dependent drift together with an arbitrarily small
real potential can produce double exponential decay of the full
spatial $L^2$ norm. We prove matching quantitative two-time lower bounds
for energy solutions of
$\partial_tu-\Delta u=vu+w\cdot\nabla u$ on $\R^n$, for $n\geq3$,
when $w\in L^\infty_{t,x}$ and either $v$ is sufficiently small in
$L^\infty_tL^{n/2}_x$ or $v\in L^\infty_tL^p_x$ for some
$n/2<p\leq\infty$, without a smallness assumption. In particular, a
nonzero global solution satisfies a double exponential lower bound
at large times. This decay scale is optimal for these coefficient
classes: In every dimension $n\geq3$, we construct a real bounded
drift, a real potential arbitrarily small in the critical space and
in any prescribed subcritical $L^p$ space, and a nonzero real energy
solution with double exponential decay. The construction localizes
Rowan's torus example by absorbing the localization error.
\end{abstract}
\maketitle
\section{Introduction}
\label{Sec01}
Advection can carry a scalar to fine spatial scales, where diffusion
acts more rapidly. How fast can this process diminish a nonzero solution
when the velocity remains uniformly bounded for all time? We study the
full spatial $L^2$ norm of solutions to
\[
 \partial_tu-\Delta u=v(t,x)u+w(t,x)\cdot\nabla u
 \quad\text{on }[0,\infty)\times\R^n
 ,
\]
where $n\geq3$.
We assume $w\in L^\infty_{t,x}$, without a divergence condition, 
allowing complex coefficients. The potential $v$
is either sufficiently small
in the scale-invariant space $L^\infty_tL^{n/2}_x$, or arbitrary in
$L^\infty_tL^p_x$ for $n/2<p\leq\infty$.
Then, solutions belong locally in time to $C_tL^2_x\cap L^2_tH^1_x$.

The influence of a flow on diffusion has been studied from several
angles. Constantin, Kiselev, Ryzhik, and Zlato\v{s}~\cite{CKRZ}
characterized incompressible flows that strongly enhance relaxation
when their amplitude increases. For bounded time-dependent flows on a
torus, Miles and Doering~\cite{MD}, using an argument of Poon~\cite{Poon}, obtained a double exponential $L^2$ lower bound.
Rowan then constructed in~\cite{Rowan} a bounded divergence-free flow on
$\mathbb T^2$ and a nonzero solution whose norm decays at this rate.
His example shows that the lower bound can be attained with a
uniformly bounded velocity; it is also the starting point of our
sharpness construction.

Our lower bounds hold on $\R^n$ for the potential classes above. In a
simple form, if $\|u(0)\|_{L^2_x}>0$,
then there are constants
$c,C>0$ such that
\[
 \|u(t)\|_{L^2_x}\geq c\exp\bigl(-C e^{C t}\bigr)
    \comma t\geq 1
 ;
\]
see Corollary~\ref{C01} below. Theorems~\ref{T01} and~\ref{T200} give
quantitative comparisons using the norms at two earlier times, with
constants uniform over solutions having those specified norms.
Two earlier norms are needed for such a comparison: For the free heat
equation, initial data of fixed $L^2$ norm concentrated at arbitrarily
high frequencies can decay arbitrarily fast over a fixed interval.
The estimates impose no spatial localization or frequency condition.

The time scale differs from the drift-free problem. When $w=0$,
the authors of the present paper obtained in~\cite{KX} the optimal
bound $ce^{-Ct}$ for small critical
potentials and $ce^{-Ct^2}$ for subcritical potentials without
smallness.
Bounded time-dependent advection permits the double
exponential scale. Other settings also reflect the role of geometry:
Huang and Xu~\cite{HX} proved exponential lower bounds along a
sequence of times for shear flows, while Nobili and Pottel~\cite{NP}
obtained algebraic lower bounds on $\R^d$ under additional decay
assumptions on the velocity.

Our double exponential bound is sharp in the stated coefficient
classes. Rowan's periodic solution is not in $L^2(\R^n)$
pointwise in time,
so we
localize it in the first two variables and add decaying factors in
the remaining dimensions. Most of the resulting zeroth-order error
can be added to a bounded drift, while the rest is supported near a
critical point of the transverse factor.
Theorem~\ref{T03} and
Corollary~\ref{C200} give, for every $n\geq3$, every $\epsilon>0$,
and any prescribed
$p\in[n/2,\infty]$
a real bounded drift and a real
potential satisfying
\[
 \|v\|_{L^\infty_tL^p_x}<\epsilon,
\]
with the drift bound independent of $\epsilon$, 
together with a
nonzero real energy solution satisfying
\[
 0<\|u(t)\|_{L^2_x}\leq C\exp(-c e^{ct})
    \comma t\geq0.
\]
The potential in this construction is generally nonzero, and the
full drift need not be divergence-free.
Thus the result establishes
sharpness
of Theorems~\ref{T01} and~\ref{T200}.

The analytic difficulty lies in making a formal frequency argument
valid for rough coefficients. The classical connection between decay
and uniqueness runs through Lax~\cite{Lax}, Lions and Malgrange~\cite{LM}, and the Hilbert-space method of Agmon and Nirenberg~\cite{AN}.
The method of Agmon and Nirenberg, lead to many results and applications, in particular to backward uniqueness; see \cite{BT,G,Z}. In particular, see \cite{K} for backward uniqueness results for parabolic equations with rough coefficients.
Also, we point out the alternative approach to backward uniqueness and time decay for equations with bounded coefficients by Lees and Protter; see~\cite{Lees,LP}.
For bounded coefficients, the Dirichlet quotient
$\|\nabla u(t)\|_2^2/\|u(t)\|_2^2$ grows at most exponentially
in the presence of a first-order term, yielding a double
exponential lower bound for the energy norm.
This elementary argument, however, does not directly extend
to the full range of potentials considered here, particularly
the small critical class $L^\infty_tL^{n/2}_x$.
The critical parabolic Carleman estimates
of Escauriaza and Vega~\cite{EV} and the Hermite-parametrix framework
of Koch and Tataru~\cite{KT} provide the appropriate background.

To absorb the drift on short intervals,
we use the mixed estimate from~\cite{KX} in similarity variables and
an integrated energy estimate.
Integration over the spatial centers removes the Gaussian weight and
yields an estimate for the full $L^2(\R^n)$ norm; iteration gives
the long-time bound. For $p>n/2$, a large-value/small-value
decomposition reduces the potential to a small critical part and a
bounded part. We prove the critical bound in Section~\ref{Sec02}, extend it to
subcritical potentials in Section~\ref{Sec03}, and give the sharpness
construction in Section~\ref{Sec04}.

\section{Lower bounds with bounded drift}
\label{Sec02}
Throughout the paper, let $n\ge3$ be fixed.
The following is our first main result
for critical potentials~$v$.

\begin{Theorem}[Bounded drift]\label{T01}
Let $I=[T_0,T_*]$ with $T_*\leq\infty$, and let
\[
u\in C(I;L^2(\R^n))\cap L^2_{\text{loc}}(I;H^1(\R^n))
\]
be a distributional solution of
\begin{equation}
 \partial_tu-\Delta u=v(t,x)u+w(t,x)\cdot\nabla u,
 \label{EQ01}
\end{equation}
where the functions $u,v,w$ may be complex-valued.
There exist positive
$\epsilon,\eta$, and $C\geq 1$ with the following properties.
Suppose
\begin{equation}
 \|v\|_{L^\infty(I;L^{n/2}(\R^n))}\leq\epsilon
 \andand M:=\|w\|_{L^\infty(I\times\R^n)}<\infty
 ,
 \label{EQ02}
\end{equation}
and for arbitrary finite $t_1<t_2<t_3$ in $I$, assume
$\|u(t_1)\|_{L^2_x},\|u(t_2)\|_{L^2_x}>0$.
Let \begin{equation}
 H=\min\{(t_2-t_1)/2,\eta/(2M^2)\}
\andand
 k=\left\lceil (t_3-t_2)/H\right\rceil,
 \label{EQ03}
\end{equation} when $M>0$, and otherwise $H=(t_2-t_1)/2$.
We have an explicit bound
\begin{align}
\begin{split}
 \|u(t_3)\|_{L^2_x}
 &\geq \|u(t_2)\|_{L^2_x}\\
 &\times\exp\left(-C^{\,k+1}
 \left(1+\log_+\frac{\|u(t_1)\|_{L^2_x}}{\|u(t_2)\|_{L^2_x}}
               +M^2(t_2-t_1)\right)\right)>0.
\end{split}\label{EQ04}
\end{align}
\end{Theorem}

Note that when $M=0$, \cite[Theorem~2.2]{KX} gives the sharper
exponential-in-time lower bound.
Above and in the sequel, we allow all constants to depend
on $n$ without mention.

Theorem~\ref{T01} follows by iterating the following estimate.

\begin{Proposition}[The short-time estimate for bounded drift]
\label{P01}
Under the hypotheses of Theorem~\ref{T01}, let
$t_1<t_2<t_3$ be finite times in~$I$. If
\[
 M^2(t_3-t_1)\leq\eta
 \andand \|u(t_1)\|_{L^2_x}\|u(t_2)\|_{L^2_x}>0,
\]
then
\begin{equation}
 \|u(t_3)\|_{L^2_x}\geq\|u(t_2)\|_{L^2_x}
 \exp\left(-C\frac{t_3-t_1}{t_2-t_1}
 \left(1+\log_+\frac{\|u(t_1)\|_{L^2_x}}{\|u(t_2)\|_{L^2_x}}\right)
 \right)>0.
 \label{EQ05}
\end{equation}
\end{Proposition}

\subsection{The energy estimate and the imported mixed estimate}

Hölder's and homogeneous Sobolev inequalities give
\begin{equation}
 |\langle vu,u\rangle_{L^2_x}|\leq C\epsilon\|\nabla u\|_{L^2_x}^2
 \andand
 |\langle w\cdot\nabla u,u\rangle_{L^2_x}|
 \leq M\|\nabla u\|_{L^2_x}\|u(t)\|_{L^2_x}.
 \label{EQ06}
\end{equation}
If
$C\epsilon\leq1/4$, taking the real part of the equation and using
Young's inequality shows
\begin{align}
 \frac12\frac{\text{d}}{\text{d}t}\|u(t)\|_{L^2_x}^2
 +\frac12\|\nabla u(t)\|_{L^2_x}^2
 &\leq CM^2\|u(t)\|_{L^2_x}^2.
   \label{EQ32}
\end{align}
To justify this, note that
$vu\in L^2_{\text{loc}}(I;H^{-1})$ and
$w\cdot\nabla u\in L^2_{\text{loc}}(I;L^2)$.
From \eqref{EQ32}, we obtain
\begin{align}
\label{EQ07}
\|u(t)\|_{L^2_x}
 &\leq e^{CM^2(t-s)}\|u(s)\|_{L^2_x}
     \comma s\leq t.
\end{align}
The inequality also holds for an energy solution continued by the free heat
flow after the terminal time.

Set $\cH=-\Delta_y+|y|^2/16$, $q=2n/(n-2)$, and $q'=2n/(n+2)$.
The following is the special case of \cite[Proposition 3.2]{KX}
that we need. Let
$I_{\text{in}}=(-1/2,\log4+1/2)$ and
$I_{\text{out}}=(-1,\log4+1)$;
there are $\kappa_0,\beta_0,C_0$ such that if
$\kappa\geq\kappa_0$, $\beta\geq\beta_0$, and
$\Psi(s)=\beta s+\tfrac\kappa2(s+2)^2+\tfrac n4s$, then
\begin{equation}
 \begin{split}
 X_\kappa(g)&:=\|g\|_{L^\infty_sL^2_y}
    +\|g\|_{L^2_sL^q_y}+\kappa^{1/4}\|g\|_{L^2_{s,y}}\\
 &\leq C
 (
   \|f_1\|_{L^1_sL^2_y}
  +\|f_2\|_{L^2_sL^{q'}_y}
  +\kappa^{-1/4}\|f_\kappa\|_{L^2_{s,y}}
 )
 \end{split}
 \label{EQ08}
\end{equation}
whenever $g\in L^2$ has compact $s$-support in
$I_{\text{in}}$ and
\[
(\partial_s+\cH-\Psi')g=f_1+f_2+f_\kappa
\] in the sense of distributions,
with the constant in \eqref{EQ08} independent of~$\beta,\kappa$. Indeed,
$\Psi'(s)=\beta+\kappa(s+2)+n/4$, which satisfies the slope
condition of that proposition on $I_{\text{out}}=(-1,\log4+1)$
when $\beta_0,\kappa_0$ are fixed sufficiently large.
We will impose the additional and essential choice $\kappa\geq\beta$.

\subsection{Absorption of the bounded first-order coefficient}
The next lemma extends \cite[Proposition~3.2]{KX} to the case of a bounded drift, relying on the energy estimate and inequalities established there for the Hermite operator.
\enew
\begin{Lemma}[Mixed Carleman estimate with bounded drift]\label{L01}
There exist $\epsilon,\delta>0$ with the following property.
Assume that 
\begin{equation}
 (\partial_s+\cH-\Psi')g
 =Vg+W(s,y)\cdot(\nabla_y+y/4)g+G,
 \label{EQ09}
\end{equation}
where
\[
g\in C_c(\R;L^2(\R^n))\cap L^2(\R;D(\cH^{1/2}))\andand \supp_s g\subset[0,\log4],
\]
with  $G\in L^1_sL^2_y$,
$\|V\|_{L^\infty_s L^{n/2}_y}\leq\epsilon$,
and
$\|W\|_{L^\infty_{s,y}}\leq\delta$.
Then
\begin{equation}
 X_\kappa(g)\leq C\|G\|_{L^1_sL^2_y}
 \label{EQ10}
\end{equation}
for 
$\kappa\geq\max\{\kappa_0,\beta\}$
and $\beta\geq\beta_0$.
\end{Lemma}

\begin{proof}[Proof of Lemma~\ref{L01}]
Sobolev and Hölder inequalities give
\\
$\|Vg\|_{L^2L^{q'}}\leq C\epsilon\|g\|_{L^2L^q}$.
Since $\|(\nabla+y/4)g\|_{L^2}\leq
2\|\cH^{1/2}g\|_{L^2_{s,y}}$, the mixed estimate
\eqref{EQ08} gives
\begin{equation}
 X_\kappa(g)\leq C\|G\|_{L^1_sL^2_y}
 +C\epsilon X_\kappa(g)
 +2C\|W\|_{L^\infty_{s,y}}\kappa^{-1/4}
          \|\cH^{1/2}g\|_{L^2_{s,y}}.
 \label{EQ11}
\end{equation}
To control the derivative term in \eqref{EQ11}, we take the real part of the $L^2_y$ pairing of~\eqref{EQ09}
with $g$, and integrate in~$s$. Since $g$ has compact $s$-support,
\[
\operatorname{Re}\int
\langle \partial_s g,g\rangle_{L^2_y}\,\text{d}s
=
\frac12\int 
\frac{\text{d}}{\text{d}s}\|g(s)\|_{L^2_y}^2\,\text{d}s
=0.
\]
Thus,
\[
\begin{split}
\|\cH^{1/2}g\|_{L^2_{s,y}}^2
\leq &
\int \Psi'(s)\|g(s)\|_{L^2_y}^2\,\text{d}s
+\int |V||g|^2\,\text{d}y\,\text{d}s \\
&+
\left|
\int W\cdot(\nabla_y+y/4)g\,\overline g\,\text{d}y\,\text{d}s
\right|
+
\left|
\int \langle G,g\rangle_{L^2_y}\,\text{d}s
\right|.
\end{split}
\]
On the support of $g$,
\[
\Psi'(s)=\beta+\kappa(s+2)+\frac n4\leq C\kappa,
\]
since $\kappa\geq\max\{\beta,\kappa_0\}$ and $\kappa_0$ is chosen
sufficiently large.
Also, by H\"older and Sobolev inequalities, we obtain
\[
\int |V||g|^2\,\text{d}y\,\text{d}s
\leq
C\epsilon\|\nabla_y g\|_{L^2_{s,y}}^2
\leq
C\epsilon\|\cH^{1/2}g\|_{L^2_{s,y}}^2.
\]
Using
\[
\|(\nabla_y+y/4)g\|_{L^2_y}
\leq
2\|\cH^{1/2}g\|_{L^2_y},
\]
we further have
\[
\left|
\int W\cdot(\nabla_y+y/4)g\,\overline g\,\text{d}y\,\text{d}s
\right|
\leq
2\|W\|_{L^\infty_{s,y}}
\|\cH^{1/2}g\|_{L^2_{s,y}}
\|g\|_{L^2_{s,y}},
\]
while
\[
\left|
\int \langle G,g\rangle_{L^2_y}\,ds
\right|
\leq
\|G\|_{L^1_sL^2_y}
\|g\|_{L^\infty_sL^2_y}.
\]
Consequently,
\begin{align}
\begin{split}
\|\cH^{1/2}g\|_{L^2_{s,y}}^2
\leq &
C\kappa\|g\|_{L^2_{s,y}}^2
+
C\epsilon\|\cH^{1/2}g\|_{L^2_{s,y}}^2 \\
&+
2\|W\|_{L^\infty_{s,y}}
\|\cH^{1/2}g\|_{L^2_{s,y}}
\|g\|_{L^2_{s,y}} 
+
\|G\|_{L^1_sL^2_y}
\|g\|_{L^\infty_sL^2_y}.
\end{split}
\label{EQ12}
\end{align}
If $C\epsilon\leq 1/4$, $\|W\|_{L^\infty_{s,y}}\leq 1$,
and $\kappa\geq1$, Young's inequality gives
\begin{align}
\label{EQ13}
\|\cH^{1/2}g\|_{L^2_{s,y}}
\leq
C\left(
\kappa^{1/2}\|g\|_{L^2_{s,y}}
+
\|G\|_{L^1_sL^2_y}^{1/2}
\|g\|_{L^\infty_sL^2_y}^{1/2}
\right).
\end{align}
This supplies the derivative control missing from the mixed estimate and
allows the bounded first-order term to be absorbed in~\eqref{EQ08}.
Inserting \eqref{EQ13} into \eqref{EQ11}, and using the definition
of $X_\kappa(g)$ and $\kappa\geq1$, we obtain
\begin{align}
\begin{split}
 X_\kappa(g)&\leq C\|G\|_{L^1_sL^2_y}
 +C(\epsilon+\|W\|_{L^\infty_{s,y}})X_\kappa(g)
 \\&\indeq
 +C\|W\|_{L^\infty_{s,y}}\kappa^{-1/4}
      \left(\|G\|_{L^1_sL^2_y}X_\kappa(g)\right)^{1/2}.
\end{split}\label{EQ14}
\end{align}
Choose $\epsilon$ and $\delta$ so that
$C(\epsilon+\delta)\leq 1/4$ and apply Young's inequality to absorb
the final term.
This proves~\eqref{EQ10}.

For energy-class data, the pairing leading to \eqref{EQ12} is justified by
time mollification and the Hilbert triple
$D(\cH^{1/2})\subset L^2\subset D(\cH^{1/2})^*$.
Indeed, $Vg\in L^2_sD(\cH^{1/2})^*$, the drift belongs to
$L^2_sL^2_y$, and $G\in L^1_sL^2_y$.
\end{proof}

\subsection{The short-time estimate and iteration}
A space-time rescaling allows the drift term to be absorbed on sufficiently short time intervals, leading to the following argument.
\enew

\begin{proof}[Proof of Proposition~\ref{P01}]
Fix $t_1<t_2<t_3$. Set $T_1=4t_3-3t_1$ and
$s_1=4(t_3-t_1)$. For $t<T_1$, define
\begin{equation}
 s(t)=\log\frac{s_1}{T_1-t}\comma
 t(s)=T_1-s_1e^{-s},\andand \rho(s)=s_1e^{-s}.
 \label{EQ15}
\end{equation}
Let $F\in C_c(J;L^2_x)\cap L^2(J;H^1_x)$ be supported
in $\{0\leq s(t)\leq\log4\}$, where $J$ is an open interval
with $\sup J<T_1$, and suppose
$(\partial_t-\Delta)F=vF+w\cdot\nabla F+G$,
$G\in L^1_tL^2_x$. For every $z\in\R^n$ set
\begin{align*}
 \begin{split}
 &\phi(s)=\beta s+\tfrac\kappa2(s+2)^2,
 \\&
 g_z(s,y)=e^{\phi(s)-|y|^2/8}
          F\left(t(s),z+\sqrt{\rho(s)}y\right),\\
 &V_z(s,y)=\rho(s)v\left(t(s),z+\sqrt{\rho(s)}y\right),\\
 &W_z(s,y)=\sqrt{\rho(s)}
     w\left(t(s),z+\sqrt{\rho(s)}y\right),
 \\&
 G_z(s,y)=\rho(s)e^{\phi(s)-|y|^2/8}
          G\left(t(s),z+\sqrt{\rho(s)}y\right).
 \end{split}
 \label{EQ16}
\end{align*}
The elementary Gaussian-conjugation calculation yields
\begin{equation}
 (\partial_s+\cH-\Psi')g_z
  =V_zg_z+W_z\cdot(\nabla_y+y/4)g_z+G_z,
 \label{EQ17}
\end{equation}
where $\Psi=\phi+ns/4$.
Since $t'(s)=\rho(s)$ and $\rho'(s)=-\rho(s)$, differentiating
the definition of $g_z$ gives
\begin{align}
  \begin{split}
\rho e^{\phi-|y|^2/8}\partial_tF
&=\partial_sg_z-\phi'g_z
  +\frac12y\cdot\nabla_yg_z+\frac{|y|^2}{8}g_z,\\
\rho e^{\phi-|y|^2/8}\Delta_xF
&=\Delta_yg_z+\frac12y\cdot\nabla_yg_z
  +\left(\frac n4+\frac{|y|^2}{16}\right)g_z,\\
\sqrt{\rho}\,e^{\phi-|y|^2/8}\nabla_xF
&=(\nabla_y+y/4)g_z,
  \end{split}
   \label{EQ30}
   \end{align}
where the derivatives of $F$ are evaluated at
$(t(s),z+\sqrt{\rho(s)}\,y)$.
Subtracting the first two equations in \eqref{EQ30} and using
$(\partial_t-\Delta_x)F=vF+w\cdot\nabla_xF+G$, we obtain
\[
\begin{split}
\left(\partial_s-\Delta_y+\frac{|y|^2}{16}
-\phi'-\frac n4\right)g_z
&=V_zg_z+W_z\cdot(\nabla_y+y/4)g_z+G_z.
\end{split}
\]
Thus, with $\cH=-\Delta_y+|y|^2/16$ and $\Psi=\phi+ns/4$,
\[
(\partial_s+\cH-\Psi')g_z
=V_zg_z+W_z\cdot(\nabla_y+y/4)g_z+G_z.
\]
For every $s,z$, scaling gives
\begin{equation}
 \begin{split}
  & \|V_z(s)\|_{L^{n/2}_y}
 =\|v(t(s))\|_{L^{n/2}_x}\leq\epsilon,
 \\&
 \|W_z\|_{L^\infty(\operatorname{supp}_s g_z\times\R^n)}
 \leq2M\sqrt{t_3-t_1}.
 \end{split}
 \label{EQ18}
\end{equation}
The assumptions on $F$ imply that $g_z$ satisfies
$g_z\in C_sL^2_y\cap L_s^2D(\mathcal{H}^{1/2})$
and $\supp_s
g\subset[0,\log4]$, as needed in the statement of
Lemma~\ref{L01}. In particular, if
$M^2(t_3-t_1)\leq\eta:=\delta^2/4$, then
Lemma~\ref{L01} applies for every $z$ with constants
independent of $z,t_1,t_2,t_3,\beta,\kappa$.
Because $s\mapsto g_z(s)$ is continuous into $L^2_y$, its
$L^\infty_sL^2_y$ bound applies at every prescribed time. Hence,
\begin{equation}
 \|g_z(s_*)\|_{L^2_y}
 \leq C\int\|G_z(s)\|_{L^2_y}\,\text ds.
 \label{EQ19}
\end{equation}
Taking the $L^2$ norm in $z$ of~\eqref{EQ19} and then applying Minkowski's
 inequality gives
\begin{align*}
 \left(\int_{\R^n}\|g_z(s_*)\|_{L^2_y}^2\,\text{d}z\right)^{1/2}
 &\leq C\left(\int_{\R^n}
       \left(\int\|G_z(s)\|_{L^2_y}\,\text{d}s\right)^2\text{d}z\right)^{1/2}\\
 &\leq C\int
       \left(\int_{\R^n}\|G_z(s)\|_{L^2_y}^2\,\text{d}z\right)^{1/2}\text{d}s.
\end{align*}
For every $Q\in L^2(\R^n)$ and $\rho>0$, Tonelli's theorem and the
translation $x=z+\sqrt\rho\,y$ leads to
\begin{align}
\label{EQ20}
\begin{split}
  &
  \int_{\R^n}\int_{\R^n}
   e^{-|y|^2/4}|Q(z+\sqrt\rho\,y)|^2\,\text{d}y\,\text{d}z
  \\&\indeq
  =\int_{\R^n}e^{-|y|^2/4}\,\text{d}y\int_{\R^n}|Q(x)|^2\,\text{d}x
  =(4\pi)^{n/2}\|Q\|_{L^2_x}^2.
\end{split}
\end{align}
Using~\eqref{EQ20} with $Q=F(t(s_*))$ and 
$Q=G(t(s))$, we obtain
\begin{align*}
 \left(\int_{\R^n}\|g_z(s_*)\|_{L^2_y}^2\,\text{d}z\right)^{1/2}
   &=(4\pi)^{n/4}e^{\phi(s_*)}\|F(t(s_*))\|_{L^2_x},\\
 \left(\int_{\R^n}\|G_z(s)\|_{L^2_y}^2\,\text{d}z\right)^{1/2}
   &=(4\pi)^{n/4}\rho(s)e^{\phi(s)}\|G(t(s))\|_{L^2_x}.
\end{align*}
Canceling $(4\pi)^{n/4}$ and using $\text{d}t=\rho(s)\,\text{d}s$ yields
\begin{align}
\label{EQ21}
 e^{\phi(s(t_*))}\|F(t_*)\|_{L^2_x}
 &\leq C\int_J e^{\phi(s(t))}\|G(t)\|_{L^2_x}\,\text{d}t
 .
\end{align}
For each $Q\in L^2_x$
and $\rho>0$, Tonelli and translation invariance show
\begin{equation}
 \int_{\R^n}\int_{\R^n}e^{-|y|^2/4}
    |Q(z+\sqrt\rho\,y)|^2\,\text{d}y\,\text{d}z
  =(4\pi)^{n/2}\|Q\|_{L^2_x}^2.
 \label{EQ20_old}
\end{equation}
Using \eqref{EQ20_old} for $F$ and $G$, and
$\,\text{d}t=\rho(s)\,\text{d}s$, we get the physical-space estimate
\begin{equation}
 e^{\phi(s(t_*))}\|F(t_*)\|_{L^2_x}
 \leq C\int_J e^{\phi(s(t))}\|G(t)\|_{L^2_x}\,\text{d}t,
 \label{EQ21_old}
\end{equation}
\colb
where $J$ is a time interval containing $[t_1, 3t_3-2t_1]$ and
$\sup J<T_1$.
If necessary, continue $u$ after $t_3$ by
$u(t)=e^{(t-t_3)\Delta}u(t_3)$, and set $v=w=0$ there.
The two one-sided $L^2$ traces agree at $t_3$, so this extension
creates no distribution supported at that time. It belongs to the
local energy class and satisfies
$\|u(t)\|_{L^2_x}\leq\|u(t_3)\|_{L^2_x}$ for $t\geq t_3$.
Set $\delta=(t_2-t_1)/8$ and choose $0\leq\chi\leq1$ smooth,
increasing from zero to one on $[t_1,t_1+\delta]$, equal to one
on $[t_1+\delta,2t_3-t_1]$, and decreasing to zero on
$[2t_3-t_1,3t_3-2t_1]$.
For
  \[F=\chi(t) u,
  \]
we get
\[
(\partial_t-\Delta)F=vF+w\cdot\nabla F+\chi'u.
\]
Apply \eqref{EQ21} to $F=\chi u$ and $G=\chi'u$,
with a time interval $J$ containing the support of $\chi$ and
$\sup J<T_1$. The energy estimate \eqref{EQ07} controls the
first transition, and the heat extension controls the second.
Since $\phi$ is increasing, and $\chi$ has total variation~1
on each transition, \eqref{EQ21} at $t_*=t_2$ gives
\begin{align}
\begin{split}
 e^{\phi(s(t_2))}\|u(t_2)\|_{L^2_x}
 &\leq C e^{\phi(s(t_1+\delta))+CM^2\delta}
                     \|u(t_1)\|_{L^2_x}
		     \\&\indeq
 +C e^{\phi(\log4)}\|u(t_3)\|_{L^2_x}.
\end{split}\label{EQ22}
\end{align}
The time gap in the weight obeys
\begin{equation}
 s(t_2)-s(t_1+\delta)
 =\int_{t_1+\delta}^{t_2}\frac{\text{d}t}{T_1-t}
 \geq\frac{7(t_2-t_1)}{32(t_3-t_1)}\andand
 0\leq s(t_2)\leq\log(4/3).
 \label{EQ23}
\end{equation}
Choose
\begin{align}
\begin{split}
 \beta&=\beta_0+
 \frac{\log_+\left(
 2Ce^{CM^2\delta}\|u(t_1)\|_{L^2_x}/
                    \|u(t_2)\|_{L^2_x}\right)}
      {s(t_2)-s(t_1+\delta)}
      \end{split}
      \end{align}
 and
 \begin{align}
 \begin{split}
 \kappa=\max\{\kappa_0,\beta\}.
\end{split}
\llabel{EQ24}
\end{align}
By the definition of $\phi$,
\[
\begin{split}
&\phi(s(t_2))-\phi(s(t_1+\delta))
\\&\indeq
=
\beta\bigl(s(t_2)-s(t_1+\delta)\bigr)
+\frac{\kappa}{2}
\left(
(s(t_2)+2)^2-(s(t_1+\delta)+2)^2
\right)
\\&\indeq
\geq
\log_+\left(
2Ce^{CM^2\delta}
\frac{\|u(t_1)\|_{L^2_x}}
     {\|u(t_2)\|_{L^2_x}}
     \right)
\end{split}
\]
since $s(t_2)>s(t_1+\delta)$ and $\kappa>0$.
Therefore,
\[
\begin{split}
e^{\phi(s(t_1+\delta))-\phi(s(t_2))}
\leq
\exp\left[
-\log_+\left(
2Ce^{CM^2\delta}
\frac{\|u(t_1)\|_{L^2_x}}
     {\|u(t_2)\|_{L^2_x}}
\right)
\right].
\end{split}
\]
Using
\[
e^{-\log_+ r}=\min\{1,r^{-1}\}
    \comma r>0,
\]
we obtain
\[
Ce^{\phi(s(t_1+\delta))+CM^2\delta}
\|u(t_1)\|_{L^2_x}
\leq
\frac12e^{\phi(s(t_2))}
\|u(t_2)\|_{L^2_x}.
\]
Thus the first term on the right-hand side of \eqref{EQ22}
can be absorbed into the left-hand side.
Hence,
\begin{equation}
 \|u(t_3)\|_{L^2_x}\geq(2C)^{-1}\|u(t_2)\|_{L^2_x}
       e^{-(\phi(\log4)-\phi(s(t_2)))}.
 \label{EQ25}
\end{equation}
By \eqref{EQ23}, $M^2(t_3-t_1)\leq\eta$, and
$(t_3-t_1)/(t_2-t_1)\geq1$, we have
\begin{equation}
 \beta+\kappa+1
 \leq C\frac{t_3-t_1}{t_2-t_1}
      \left(1+\log_+\frac{\|u(t_1)\|_{L^2_x}}
                        {\|u(t_2)\|_{L^2_x}}\right).
 \label{EQ26}
\end{equation}
As $\phi(\log4)-\phi(s(t_2))\leq C(\beta+\kappa)$,
\eqref{EQ25} proves~\eqref{EQ05}.
\end{proof}

We now iterate the short-time estimate to obtain the global bound.

\begin{proof}[Proof of Theorem~\ref{T01}]
Let $H$ and $k$ be as defined in the statement of the theorem.
Set $h=(t_3-t_2)/k$ and $\tau_j=t_2+jh$ for
$-1\leq j\leq k$. Then
\[
0<h\leq H\leq\frac{t_2-t_1}{2}
\comma
t_1<\tau_{-1}<\tau_0=t_2
    ,\andand
\tau_k=t_3.
\]
By the energy estimate \eqref{EQ07},
\begin{align*}
\begin{split}
&
0<\|u(t_2)\|_{L^2_x}
 \leq e^{CM^2h}\|u(\tau_{-1})\|_{L^2_x},
\end{split}
\end{align*}
with
\begin{align*}
\begin{split}
\|u(\tau_{-1})\|_{L^2_x}
 \leq e^{CM^2(\tau_{-1}-t_1)}\|u(t_1)\|_{L^2_x}
\end{split}
\end{align*}
and
\begin{align}
\begin{split}
 \log_+\frac{\|u(\tau_{-1})\|_{L^2_x}}
                 {\|u(\tau_0)\|_{L^2_x}}
 &\leq
 \log_+\frac{\|u(t_1)\|_{L^2_x}}
                 {\|u(t_2)\|_{L^2_x}}
 +CM^2(t_2-t_1).
\end{split}\label{EQ27}
\end{align}
Every triple $(\tau_{j-1},\tau_j,\tau_{j+1})$, $0\leq j<k$,
has equal gaps $h$ and satisfies
\[
M^2(\tau_{j+1}-\tau_{j-1})=2M^2h\leq\eta.
\]
Proposition~\ref{P01} therefore shows, inductively,
that for $0\leq j<k$ we have
\begin{align}
\begin{split}
 1+\log_+\frac{\|u(\tau_j)\|_{L^2_x}}
                 {\|u(\tau_{j+1})\|_{L^2_x}}
 &\leq (1+C)\left(
 1+\log_+\frac{\|u(\tau_{j-1})\|_{L^2_x}}
                  {\|u(\tau_j)\|_{L^2_x}}
 \right).
\end{split}\label{EQ28}
\end{align}
Iterating \eqref{EQ28} and summing, after increasing $C>1$
if necessary, we obtain
\begin{align}
\begin{split}
 \log\frac{\|u(t_2)\|_{L^2_x}}{\|u(t_3)\|_{L^2_x}}
 &=\sum_{j=0}^{k-1}
   \log\frac{\|u(\tau_j)\|_{L^2_x}}
                 {\|u(\tau_{j+1})\|_{L^2_x}}             
    \leq\sum_{j=0}^{k-1}\left(1+
    \log_+\frac{\|u(\tau_j)\|_{L^2_x}}
                  {\|u(\tau_{j+1})\|_{L^2_x}}\right)\\
 &\leq C^{\,k+1}\left(
 1+\log_+\frac{\|u(\tau_{-1})\|_{L^2_x}}
                  {\|u(\tau_0)\|_{L^2_x}}
 \right).
\end{split}\label{EQ29}
\end{align}
By combining \eqref{EQ27} and \eqref{EQ29}, and increasing $C$
if necessary, we then obtain~\eqref{EQ04}.
\end{proof}

\section{Subcritical potentials without a smallness assumption}
\label{Sec03}
The following statement is our first main result for bounded drift $w$ and subcritical potential $v$.

\begin{Theorem}[Bounded drift and subcritical potential]\label{T200}
Let $p\in(n/2,\infty]$, and $I=[T_0,T_*]$ with $T_*\leq\infty$.
Suppose that
\[
u\in C(I;L^2(\R^n))\cap L^2_{\text{loc}}(I;H^1(\R^n))
\]
solves \eqref{EQ01} in the sense of distributions, where $u,v,w$ may be complex-valued and
\[
 N:=\|v\|_{L^\infty(I;L^p(\R^n))}<\infty
 \andand
 M:=\|w\|_{L^\infty(I\times\R^n)}<\infty.
\]
With the constant $\epsilon$ from Lemma~\ref{L01}, set
\begin{equation}
 R=\begin{cases}
 \epsilon^{-n/(2p-n)}N^{2p/(2p-n)},&p<\infty,\\
 N,&p=\infty.
 \end{cases}
 \label{EQ200}
\end{equation}
There are constants $\tilde{\eta}>0$ and $C'>1$, depending only on $n$,
such that the following holds. For finite $t_1<t_2<t_3$ in $I$ with
$\|u(t_1)\|_{L^2_x}\|u(t_2)\|_{L^2_x}>0$, let
\begin{equation}
 H=\begin{cases}
 (t_2-t_1)/2,&M^2+R=0,\\[2pt]
 \min\{(t_2-t_1)/2,\tilde{\eta}/(2(M^2+R))\},&M^2+R>0,
 \end{cases}
 \label{EQ201}
\end{equation}
and set $k=\left\lceil(t_3-t_2)/H\right\rceil$.
Then,
\begin{align}
\begin{split}
 \|u(t_3)\|_{L^2_x}
 &\geq \|u(t_2)\|_{L^2_x}
 \\&\!\!\times
 \exp\biggl(-(C')^{k+1}\left(
  1+\log_+\frac{\|u(t_1)\|_{L^2_x}}{\|u(t_2)\|_{L^2_x}}
  +(M^2+R)(t_2-t_1)\right)\biggr)>0.
\end{split}\label{EQ202}
\end{align}
\end{Theorem}
When $M=0$, \cite[Theorem~2.1]{KX} gives a sharper
Gaussian-type lower bound for $p>n/2$.

\begin{proof}[Proof of Theorem~\ref{T200}]
If $N=0$, the result follows from Theorem~\ref{T01}, so assume $N>0$.
If $p=\infty$, set $v_c=0$ and $v_b=v$,
while if $p<\infty$, set
\[
 v_c=v\mathbf1_{\{|v|>R\}}
 \andand
 v_b=v-v_c.
\]
For finite $p$, the choice of $R$ gives, for almost every $t$,
\[
 \|v_c(t)\|_{L^{n/2}_x}^{n/2}
 \leq R^{n/2-p}\|v(t)\|_{L^p_x}^{p}
 \leq\epsilon^{n/2}.
\]
Also, $\|v_b\|_{L^\infty_{t,x}}\leq R$ by construction.
For $p=\infty$, these bounds follow from $v_c=0$, $v_b=v$,
and $R=N$. Thus, in both cases
\begin{equation}
 \|v_c(t)\|_{L^{n/2}_x}\leq\epsilon
 \andand
 \|v_b\|_{L^\infty_{t,x}}\leq R.
 \label{EQ203}
\end{equation}
Therefore, the energy estimate \eqref{EQ07}, with $v_c$ absorbed by
Sobolev inequality and $v_b$ treated as a bounded zeroth-order term, becomes
\begin{equation}
 \|u(t)\|_{L^2_x}
 \leq e^{C(M^2+R)(t-s)}\|u(s)\|_{L^2_x}
     \comma s\leq t.
 \label{EQ204}
\end{equation}
It remains only to modify the proof of Proposition~\ref{P01}.
In the similarity variables used there, writing
  \begin{align}
  \begin{split}
   & V_{c,z}(s,y)
 =\rho(s)v_c(t(s),z+\sqrt{\rho(s)}y),
 \\&
 V_{b,z}(s,y)
 =\rho(s)v_b(t(s),z+\sqrt{\rho(s)}y),
  \end{split}
   \label{EQ31}  
  \end{align}
the transformed equation becomes
\[
 (\partial_s+\cH-\Psi')g_z
 =
 V_{c,z}g_z
 +W_z\cdot(\nabla_y+y/4)g_z
 +V_{b,z}g_z+G_z.
\]
By the change of variables $x=z+\sqrt{\rho(s)}y$ and \eqref{EQ203},
we have
\[
 \|V_{c,z}(s)\|_{L^{n/2}_y}
 =
 \|v_c(t(s))\|_{L^{n/2}_x}
 \leq\epsilon.
\]
Moreover, since $\rho(s)\leq4(t_3-t_1)$ on $0\leq s\leq\log4$,
\[
 \|V_{b,z}\|_{L^\infty_{s,y}}
 \leq4R(t_3-t_1).
\]
Thus Lemma~\ref{L01}, with $G_z+V_{b,z}g_z$ as the source, gives
\[
\begin{aligned}
 X_\kappa(g_z)
 &\leq C\|G_z\|_{L^1_sL^2_y}
     +C\|V_{b,z}g_z\|_{L^1_sL^2_y}\\
 &\leq C\|G_z\|_{L^1_sL^2_y}
     +4C(\log4)R(t_3-t_1)
       \|g_z\|_{L^\infty_sL^2_y}\\
 &\leq C\|G_z\|_{L^1_sL^2_y}
     +4C(\log4)R(t_3-t_1)X_\kappa(g_z).
\end{aligned}
\]
Choose $0<\tilde{\eta}\leq\eta$ sufficiently small so that
$4C(\log4)\tilde{\eta}\leq1/2$. The last term is then absorbed whenever
\[
 (M^2+R)(t_3-t_1)\leq\tilde{\eta}.
\]
The remainder of the proof of Proposition~\ref{P01} is unchanged, except
that \eqref{EQ204} is used in place of~\eqref{EQ07}. Therefore,
\[
 \|u(t_3)\|_{L^2_x}
 \geq\|u(t_2)\|_{L^2_x}
 \exp\left(
 -C\frac{t_3-t_1}{t_2-t_1}
 \left(
 1+\log_+
 \frac{\|u(t_1)\|_{L^2_x}}{\|u(t_2)\|_{L^2_x}}
 \right)\right)
\]
when $(M^2+R)(t_3-t_1)\leq\tilde{\eta}$.

Finally, we repeat verbatim the iteration in the proof of
Theorem~\ref{T01}, with $M^2$ replaced by $M^2+R$ and with $H,k$ chosen
as in~\eqref{EQ201}, leading to~\eqref{EQ202}.
\end{proof}

\begin{Corollary}[Long-time lower bound from one nonzero time]\label{C01}
Let $I=[T_0,\infty)$, and assume that
\[
 u\in C(I;L^2(\R^n))\cap L^2_{\mathrm{loc}}(I;H^1(\R^n))
\]
solves \eqref{EQ01} with $\|u(T_0)\|_{L^2_x}>0$.
Suppose $w\in L^\infty(I\times\R^n)$ and assume either that
$\|v\|_{L^\infty(I;L^{n/2}(\R^n))}\leq\epsilon$ or 
$v\in L^\infty(I;L^p(\R^n))$ for some $n/2<p\leq\infty$.
Then there exist $c,C>0$, independent of $t$, such that
\[
 \|u(t)\|_{L^2_x}\geq c\exp\bigl(-C e^{C(t-T_0)}\bigr)
     \comma 
 t\geq T_0+1.
\]
\end{Corollary}

\begin{proof}[Proof of Corollary~\ref{C01}]
By continuity, $\|u(t_2)\|_{L^2_x}>0$ for some
$t_2\in(T_0,T_0+1)$. Then apply Theorem~\ref{T01} in the critical
case or Theorem~\ref{T200} in the subcritical case with
$t_1=T_0$ and this fixed~$t_2$. The resulting two-time estimate
gives the stated double exponential lower bound.
\end{proof}

\section{Whole-space sharpness for the bounded endpoint}
\label{Sec04}
The following is our second main result which demonstrates
the sharpness of Theorem~\ref{T01}.

\begin{Theorem}[Whole-space sharpness in the critical-potential class]
\label{T03}
Let $\epsilon>0$ be arbitrary. There are real-valued
$w\in L^\infty([0,\infty)\times\R^n;\R^n)$ and
$v\in L^\infty([0,\infty);L^{n/2}(\R^n;\R))$, with
$\|v\|_{L^\infty_tL^{n/2}_x} \leq \epsilon$, and a real-valued
\[
u\in C([0,\infty);L^2(\R^n))
   \cap L^2_{\text{loc}}([0,\infty);H^1(\R^n))
\]
such that
\begin{equation}
 \partial_tu-\Delta u=w\cdot\nabla u+vu
 \label{EQ62}
\end{equation}
and, for some $c,C>0$ and all $t\geq0$,
\begin{equation}
 0<\|u(t)\|_{L^2_x}\leq C\exp(-c e^{ct}).
 \label{EQ63}
\end{equation}
The bound on $\|w\|_{L^\infty_{t,x}}$ is independent of~$\epsilon$.
\end{Theorem}

\begin{proof}[Proof of Theorem~\ref{T03}]
Write $\R^n=\R^2_y\times\R^k_z$, where $k=n-2\geq1$.
By \cite[Theorem~1.1]{Rowan}, there are a bounded real
divergence-free field $B(t,y)$ on $\mathbb T^2$ and a solution
$\theta$ with smooth nonzero initial data such that
\[
 \partial_t\theta-\Delta_y\theta=B\cdot\nabla_y\theta
 \andand
 \|\theta(t)\|_{L^2(\mathbb T^2)}
 \leq C_0e^{-c_0e^{c_0t}}\|\theta(0)\|_{L^2(\mathbb T^2)}.
\]
If the selected solution is complex, take a nonzero real or imaginary
part and adjust~$C_0$. Extend $B$ and $\theta$ periodically to
$\R^2$, and set $M=\|B\|_{L^\infty_{t,y}}$.

We first localize the periodic solution in~$y$. For $L\geq1$, set
\[
 h_L(y)=\frac{L^2}{L^2+|y|^2}.
\]
Then $h_L\in L^2(\R^2)$, and direct differentiation gives
\[
 \frac{\nabla_y h_L}{h_L}=-\frac{2y}{L^2+|y|^2}\andand
 \frac{\Delta_yh_L}{h_L}
 =-\frac4{L^2+|y|^2}
  +\frac{8|y|^2}{(L^2+|y|^2)^2}.
\]
The product rule, followed by
$\nabla_y(h_L\theta)=h_L\nabla_y\theta+\theta\nabla_yh_L$,
therefore yields
\begin{align}
\begin{split}
\label{EQ401}
 (\partial_t-\Delta_y)(h_L\theta)
 &=h_LB\cdot\nabla_y\theta
   -2\nabla_yh_L\cdot\nabla_y\theta
   -\theta\Delta_yh_L\\
 &=\left(B+\frac{4y}{L^2+|y|^2}\right)
       \cdot\nabla_y(h_L\theta)+E_L(t,y)h_L\theta,
\end{split}
\end{align}
where the localization error has the explicit form
\[
 E_L(t,y)=\frac{4+2B(t,y)\cdot y}{L^2+|y|^2}
 \andand
 |E_L(t,y)|\leq\frac{4+2M}{\sqrt{L^2+|y|^2}}.
\]
Next, we choose a decaying factor in the remaining $k$ directions:
\[
 q(z)=e^{-\langle z\rangle}
 ,
\]
where
$ \langle z\rangle=(1+|z|^2)^{1/2}$,
which belongs to $H^1(\R^k)$. We also have
\[
 \frac{\nabla_z q}{q}=\nabla_z\log q=-\frac z{\langle z\rangle}\andand
 \frac{\Delta_zq}{q}
 =-\frac{k}{\langle z\rangle}
   +\frac{|z|^2}{\langle z\rangle^3}
   +\frac{|z|^2}{\langle z\rangle^2}.
\]
At $z=0$, the gradient of $q$ vanishes and
$\Delta_zq(0)=-kq(0)$; this determines the time factor
$e^{-kt}$ below. Define the transverse drift
\[
 b(z)=\left(
   \frac{k}{1+\langle z\rangle}
  +\frac1{\langle z\rangle}
  +\frac1{\langle z\rangle^2}\right)z.
\]
The preceding identities and
$|z|^2=(\langle z\rangle-1)(\langle z\rangle+1)$ give
\begin{align}
\label{EQ402}
 -kq-\Delta_zq=b\cdot\nabla_zq\andand
 \|b\|_{L^\infty_z}\leq k+2.
\end{align}
By \eqref{EQ401} and \eqref{EQ402}, for
\[
 u(t,y,z)=e^{-kt}h_L(y)q(z)\theta(t,y),
\]
we have
\begin{align*}
 &
 \partial_t u
 =e^{-kt}q\,\partial_t(h_L\theta)
   -ke^{-kt}h_Lq\theta
   ,
  \\&
 \Delta_{y,z}u
 =e^{-kt}q\,\Delta_y(h_L\theta)
   +e^{-kt}h_L\theta\,\Delta_zq.
\end{align*}
Hence,
\begin{align*}
 (\partial_t-\Delta_{y,z})u
 &=
 e^{-kt}q(\partial_t-\Delta_y)(h_L\theta)
 +e^{-kt}h_L\theta(-kq-\Delta_zq)\\
 &=
 \left(B+\frac{4y}{L^2+|y|^2},\,b\right)
 \cdot\nabla_{y,z}u+E_Lu.
\end{align*}
To turn most of $E_Lu$ into a bounded drift, we must avoid
$z=0$, where $\nabla_zq=0$. Choose
$\chi\in C^\infty([0,\infty);[0,1])$ equal to one on $[0,1]$
and to zero on~$[2,\infty)$. Set
\[
\begin{aligned}
 &
 v(t,y,z)
  =E_L(t,y)\chi\left(\sqrt{L^2+|y|^2}\,|z|\right),\\&
 w_y(t,y,z)=B(t,y)+\frac{4y}{L^2+|y|^2},
\end{aligned}
\]
and, for $z\ne0$,
\[
\begin{aligned}
 w_z(t,y,z)
  &=b(z)+E_L(t,y)
      \left(1-\chi\bigl(\sqrt{L^2+|y|^2}\,|z|\bigr)\right)
      \frac{\nabla_z\log q(z)}{|\nabla_z\log q(z)|^2}.
\end{aligned}
\]
Set $w_z(t,y,0)=0$ and $w=(w_y,w_z)$. 
The factor $1-\chi$ makes the correction vanish near $z=0$, where
$\nabla_z\log q=0$.  Since
\[
 \nabla_z u=u\nabla_z\log q,
\]
we have
\begin{align*}
 (w_z-b)\cdot\nabla_z u
 &=
 E_L(t,y)
 \left(1-\chi\left(\sqrt{L^2+|y|^2}\,|z|\right)\right)u.
\end{align*}
On the other hand,
\[
 vu=
 E_L(t,y)
 \chi\left(\sqrt{L^2+|y|^2}\,|z|\right)u.
\]
Hence,
\[
 (w_z-b)\cdot\nabla_z u+vu=E_Lu,
\]
which proves~\eqref{EQ62}.
The drift is bounded uniformly in~$L$.  Indeed,
\[
 |w_y|
 \leq M+\frac{4|y|}{L^2+|y|^2}
 \leq M+\frac{2}{L}.
\]
Also, $\|b\|_{L^\infty_z}\leq k+2$.  Since $1-\chi(r)=0$ for
$r\leq1$ on the set where $w_z-b\neq0$, we have
$\sqrt{L^2+|y|^2}\,|z|>1$.
Hence,
\begin{align*}
 |w_z-b|
 &\leq
 |E_L|\frac{1}{|\nabla_z\log q|}
 =
 |E_L|\frac{\langle z\rangle}{|z|}\\
 &\leq
 C(1+M)
 \left(
 \frac{1}{\sqrt{L^2+|y|^2}}
 +
 \frac{1}{\sqrt{L^2+|y|^2}\,|z|}
 \right)
 \leq C(1+M).
\end{align*}
Therefore,
\[
 \|w\|_{L^\infty_{t,x}}\leq C(1+M),
\]
with a constant independent of~$L$.
On the other hand, as $k=n-2$,
the bound on $E_L$ gives
\begin{align*}
 \|v(t)\|_{L^{n/2}(\R^n)}^{n/2}
 &\leq C(1+M)^{n/2}
    \int_{\R^2}\int_{\{|z|\leq 2/\sqrt{L^2+|y|^2}\}}(L^2+|y|^2)^{-n/4}
    \,\text{d}z\,\text{d}y\\
 &\leq C(1+M)^{n/2}
    \int_{\R^2}(L^2+|y|^2)^{-(n/2+k)/2}\,\text{d}y\\
 &=\frac{2\pi C(1+M)^{n/2}}{n/2+k-2}
      L^{\,2-n/2-k}
 \leq C(1+M)^{n/2}L^{\,4-3n/2},
\end{align*}
uniformly in $t$.
Therefore,
\[
 \|v\|_{L^\infty_tL^{n/2}_x}
 \leq C(1+M)L^{\,8/n-3}\to0
  \asas L\to\infty.
\]
Choose $L$ sufficiently large that
$\|v\|_{L^\infty_tL^{n/2}_x}<\epsilon$
and the smallness hypothesis of Theorem~\ref{T01} is satisfied.

It remains to check the decay and nonvanishing of this same
whole-space solution. Periodicity yields the exact identity
\[
 \|u(t)\|_{L^2(\R^n)}^2
 =e^{-2kt}\|q\|_{L^2(\R^k)}^2
  \int_{\mathbb T^2}|\theta(t,\bar y)|^2
   \sum_{m\in\mathbb Z^2}
      h_L(\bar y+2\pi m)^2\,\text{d}\bar y.
\]
For each fixed $L$, the series
\[
 \sum_{m\in\mathbb Z^2}h_L(\bar y+2\pi m)^2
\]
converges uniformly on $\mathbb T^2$, since
$h_L(y)^2=O_L(|y|^{-4})$. Hence,
\[
 \sup_{\bar y\in\mathbb T^2}
 \sum_{m\in\mathbb Z^2}h_L(\bar y+2\pi m)^2
 \leq C_L
\]
for some $C_L<\infty$.  Therefore,
\[
 \|u(t)\|_{L^2(\mathbb R^n)}^2
 \leq
 C_L e^{-2kt}\|q\|_{L^2(\mathbb R^k)}^2
 \|\theta(t)\|_{L^2(\mathbb T^2)}^2,
\]
which, together with
\[
 \|\theta(t)\|_{L^2(\mathbb T^2)}
 \leq
 C_0e^{-c_0e^{c_0t}}
 \|\theta(0)\|_{L^2(\mathbb T^2)},
\]
gives the upper bound in~\eqref{EQ63}.

Since $\theta(0)\neq0$ and both $h_L$ and $q$ are strictly positive, $\|u(0)\|_{L^2(\mathbb R^n)}>0$.
By continuity, $\|u(t_2)\|_{L^2_x}>0$ for some $t_2>0$.
For $0\leq t<t_2$, \eqref{EQ07} gives $\|u(t)\|_{L^2_x}>0$;
for $t>t_2$, apply Theorem~\ref{T01} with times $(0,t_2,t)$.
Thus
\[
 \|u(t)\|_{L^2(\mathbb R^n)}>0
\]
for every finite $t\geq0$.
\end{proof}

\begin{Corollary}[Sharpness in every subcritical class]\label{C200}
For $p\in(n/2,\infty]$, and any $\epsilon>0$, the real-valued
coefficients and energy solution of Theorem~\ref{T03} can be chosen with
$\|v\|_{L^\infty_tL^p_x}<\epsilon$. 
The bound on
$\|w\|_{L^\infty_{t,x}}$ is independent of $\epsilon$, and
\eqref{EQ63} still holds.
\end{Corollary}

\begin{proof}[Proof of Corollary~\ref{C200}]
In the construction of Theorem~\ref{T03}, the support of $v(t,y,\cdot)$
lies in $\{|z|\leq2/\sqrt{L^2+|y|^2}\}$. If $p=\infty$, the
explicit expression for $E_L$ gives
\[
 \|v\|_{L^\infty_{t,x}}
 \leq\frac{4}{L^2}+\frac{M}{L}\to0
\asas L\to\infty
.
\]
For finite $p$, with $k=n-2$, we have
\begin{align}
\begin{split}
 \|v(t)\|_{L^p(\R^n)}^p
 &\leq C_{n,p}(1+M)^p
    \int_{\R^2}(L^2+|y|^2)^{-(p+k)/2}\,\text{d}y\\
 &=\frac{2\pi C_{n,p}(1+M)^p}{p+n-4}
       L^{\,4-n-p}.
\end{split}\label{EQ207}
\end{align}
Here, 
$p+n-4>0$. 
In either case, choosing $L$ sufficiently large
gives the stated small $L^p$ norm while preserving the drift bound
and the nonvanishing solution in Theorem~\ref{T03}.
\end{proof}

\colb
\section*{Acknowledgments} 
IK and QX were supported in part by the NSF grant DMS-2205493
and the Simons grant SFI-MPS-TSM-00014233.
The authors used ChatGPT for language editing. All mathematical content was written and independently verified by the authors.

\small
\medskip
\noindent
I.~Kukavica,
{Department of Mathematics, University of Southern California, Los Angeles, CA 90089}, e-mail: kukavica@usc.edu

\medskip
\noindent
Q.~Xu, {Department of Mathematics, University of Southern California, Los Angeles, CA 90089}, e-mail: xuqi@usc.edu

\end{document}